\documentclass[a4paper,10pt]{amsart}

\usepackage{amsmath, amssymb, amsthm,  amscd}
\usepackage{tikz}

\newtheorem{defn}{Definition}
\newtheorem{thm}[defn]{Theorem}
\newtheorem{rmk}[defn]{Remark}

\newtheorem{lemma}[defn]{Lemma}
\newtheorem{prop}[defn]{Proposition}
\newtheorem{cor}[defn]{Corollary}
\numberwithin{defn}{section}

\newcommand{\isom}{\cong}

\newcommand{\del}{\partial}
\newcommand{\less}{\backslash}

\newcommand{\Z}{\mathbb{Z}}
\newcommand{\R}{\mathbb{R}}
\newcommand{\C}{\mathbb{C}}
\newcommand{\Pb}{\mathbb{P}}
\newcommand{\F}{\mathcal{F}}
\newcommand{\T}{\mathcal{T}}
\newcommand{\M}{\mathcal{M}_{1,1}}

\newcommand{\N}{\mathbb{N}}
\newcommand{\D}{\mathcal{D}}
\newcommand{\A}{\mathcal{A}}

\newcommand{\E}{\mathcal{E}}
\newcommand{\h}{\mathfrak{h}}
\newcommand{\Br}{\mathrm{Br}}
\newcommand{\EG}{\mathrm{EG}}

\newcommand{\Abar}{\bar{\A}}
\newcommand{\sigmabar}{\overline{\sigma}}

\newcommand{\Stab}{\operatorname{Stab}}
\newcommand{\Hom}{\operatorname{Hom}}

\newcommand{\Mod}{\operatorname{mod}}

\newcommand{\Sph}{\operatorname{Sph}}
\newcommand{\Ext}{\operatorname{Ext}}
\newcommand{\Cone}{\operatorname{Cone}}
\newcommand{\Aut}{\operatorname{Aut}}

\newcommand{\CY}{\mathrm{CY}}
\newcommand{\SL}{\mathrm{SL}}
\newcommand{\SU}{\mathrm{SU}}
\newcommand{\PSL}{\mathrm{PSL}}

\newcommand{\xxx}[1]{}

\address{Mathematical Institute, Oxford, OX1 3LB}
\email{sutherlandt@maths.ox.ac.uk}

\title[The modular curve as stability conditions on a $\CY_3$ algebra]{The modular curve as the space of stability conditions of a $\CY_3$ algebra}
\author{Tom Sutherland}

\begin{document}
\begin{abstract}
We prove that a connected component of the space of stability conditions of a $\CY_3$ triangulated category generated by an $A_2$-collection of $3$-spherical objects is isomorphic to the universal cover of the $\C^*$-bundle of non-zero holomorphic differentials on the moduli space of elliptic curves.
\end{abstract} 
 
\maketitle

\section{Introduction}

The space of stability conditions $\Stab(\D)$ of a triangulated category $\D$ was introduced in \cite{B}.  As a set it has a description as the pairs $(\A, Z)$ where $\A$ is the heart of a t-structure on $\D$, and $Z: K(\A) \isom K(\D) \to \C$ is a stability function on $\A$ known as the central charge.  As the forgetful map $\Stab(\D) \to \Hom(K(\D), \C)$ remembering just the central charge is a local homeomorphism \cite[Prop 6.3]{B}, $\Stab(\D)$ has the structure of a complex manifold.  It carries an action of the group of autoequivalences $\Aut(\D)$ and a free action of $\C$ for which $\Z \subset \C$ acts as the autoequivalence $[1]$, the shift functor of $\D$.

In this paper we compute a connected component $\Stab^0(\D)$ of the space of stability conditions of $\D = \D_{fd}(G A_2)$, the derived category of finite dimensional modules over the Ginzburg dg algebra of the $A_2$ quiver.  This is a $\CY_3$ triangulated category generated (cf \cite[Sect 2]{KYZ}) by two objects $S$ and $T$ with
\[
\Hom(S,S) \isom \C \isom \Hom(T,T) \qquad \Ext^1(S,T) \isom \C
\]
We will call the heart $\A^0$ consisting of all modules supported in degree zero the standard heart.  It is equivalent to the abelian category of finitely generated modules over the path algebra of the $A_2$ quiver, and its two simple objects are $S$ and $T$.  We study the connected component $\Stab^0(\D)$ which contains stability conditions supported on the standard heart $\A^0$.

In section two we study the subquotient $\Aut^0(\D)$ of $\Aut(\D)$ of those autoequivalences preserving the connected component $\Stab^0(\D)$ modulo those which act trivially on it.  We show that the set of hearts supporting a stability condition in $\Stab^0(\D)$ is an $\Aut^0(\D)$-torsor and deduce that
\begin{thm} \label{aut}
$\Aut^0(\D)$ is isomorphic to the braid group $\Br_3$ on three strings.
\end{thm}

In section three we show how to define central charges using periods of a meromorphic differential $\lambda$ on the universal family of framed elliptic curves $\E \to \widetilde{\M}$.  Restricted to a fibre $E$, $\lambda$ has a single pole of order $6$ at the marked point $p$ and double zeroes at each of the half-periods.  Using the framing $\{\alpha, \beta\}$ and the basis $\{[S], [T]\}$ of $K(\D)$ to identify the lattices $H_1(E \less p, \Z) \isom K(\D)$, we prove

\pagebreak
\begin{thm} \label{stab}
 There is a biholomorphic map
\begin{center}
\begin{tikzpicture}[node distance=2cm, auto]
  \node (A) {$\widetilde{\M}$};
  \node (B) [right of=A] {$\Stab^0(\D)/ \C$};
  \node (D) [below of=B] {$\Pb\Hom(K(\D),\C)$};
;
  \draw[->] (A) to node {$f$} (B);
  \draw[->] (A) to node [anchor=east] {$[\int_\alpha \lambda : \int_\beta \lambda] \: \:$} (D);
    \draw[->] (B) to node {$[Z(S):Z(T)]$} (D);
\end{tikzpicture}
\end{center}
 lifting the period map of $\lambda$.  It is equivariant with respect to the actions of $\PSL(2,\Z)$ on the left by deck transformations and on the right by $\Aut^0(\D)/ \Z$ which are both determined by their induced actions on $K(\D)$.
\end{thm}
 
As a corollary we obtain a $\Br_3$-equivariant biholomorphism from the universal cover of the $\C^*$-bundle $L^*$ of non-zero holomorphic differentials on $\M$ to $\Stab^0(\D)$.

\begin{rmk} \label{painleve}
In \cite{vdPS} the authors list 9 families of rank two connections on $\Pb^1$ having at least one irregular singularity  which have precisely a one-parameter family of isomonodromic deformations described by one of the Painlev\'{e} equations.  To each such family we associate a quiver $Q$ as in \cite{GMN}, where $Q = A_2$ corresponds to the family whose isomonodromic deformations are given by solutions to the first Painlev\'{e} equation.   It is anticipated that similar considerations to those of this paper will give a description of the space of numerical stability conditions of $\D_{fd}(GQ)$ as the universal cover of a $\C^*$-bundle of meromorphic differentials over a moduli space of elliptic curves.  We intend to return to this in future work.
\end{rmk}

The author would like to thank his PhD supervisor Tom Bridgeland for suggesting the problem and for many helpful discussions, and the EPSRC for financial support.

\section{Autoequivalences}

In this section we prove Theorem \ref{aut}.  We show that every heart supporting a stability condition in $\Stab^0(\D)$ is a translate of the standard heart $\A^0 = \Mod(\C A_2)$ by a composite of a spherical twist and the shift functor $[1]$.  We deduce that every element of $\Aut^0(\D)$ is expressible in this way.  The group of spherical twists $\Sph(\D)$ is a subgroup of $\Aut^0(\D)$ of index five, and we use a result of Seidel-Thomas that $\Sph(\D) \isom \Br_3$ to deduce that $\Aut^0(\D) \isom \Br_3$, the braid group on three strings.

\begin{defn}
An object $X \in \D$ is spherical if $\Hom_{\D}(X, X) \isom \C \oplus \C[-3]$.  For $X$ spherical there is a twist functor $\Phi_X$ such that
\[
 \Phi_X(Y) = \Cone(X \otimes \Hom(X,Y) \to Y)
\]
\end{defn}
There are two spherical objects $S$ and $T$ in $\D$ which are the simple objects in the standard heart $\A^0$. They form an $A_2$-collection \cite[Def 1.1]{ST} as $\Ext^1(S,T) \isom \C$.

\begin{thm} \cite[Thms 1.2, 1.3]{ST}
 The spherical twists $\Phi_S, \Phi_T$ satisfy the braid relations
\[
 \Phi_S \Phi_T \Phi_S = \Phi_T \Phi_S \Phi_T
\]
 and generate a subgroup $\Sph(\D)$ of the group of autoequivalences $\Aut(\D)$ isomorphic to the braid group on three strings $\Br_3$. 
\end{thm}

The braid group $\Br_3$ has the following presentation by generators and relations \cite[Sect 1.14]{KT}
\[
 \langle \sigma_1, \sigma_2 \: |  \: \sigma_1 \sigma_2 \sigma_1 = \sigma_2 \sigma_1 \sigma_2 \rangle
\]
Its centre is the infinite cyclic subgroup generated by the element $u = (\sigma_1 \sigma_2)^3$ \cite[Thm 1.24]{KT} giving us the short exact sequence
\[
 1 \to \Z \to \Br_3 \to \PSL(2, \Z) \to 1
\]
where the quotient map sends the generators $\sigma_1, \sigma_2$ to 
\[
\left ( \begin{array}{cc} 1 & 1 \\
                           0 & 1 \\
 \end{array} \right )
\qquad \left ( \begin{array}{cc} 1 & 0 \\
                                -1 & 1\\
 \end{array} \right )
\]
We note that the action of a spherical twist $\Phi_X$ on $K(\D)$ is given by the formula
\[
 \Phi_X([Y]) = [Y] - \chi(X,Y) [X]
\]
and so the map $\Sph(\D) \to \PSL(2,\Z)$ sends a spherical twist to the matrix given by its action on the lattice $K(\D)$ with respect to the basis $\{[S], [T]\}$.

We now study the combinatorial backbone of the space of stability conditions, namely (a connected component of) the \emph{exchange graph} of hearts of $\D$.

\begin{defn}  We say $\A'$ is a simple tilt of $\A$ at $S$ if either
\begin{itemize}
\item $\A'$ is the left tilt of $\A$ with respect to the torsion pair 
\[
 \T = \langle S \rangle = \{ S^{\oplus n} \: | \: n \in \N_0\} \qquad \F = \{X \: | \: \Hom_{\A}(S,X)=0\}
\]
\item $\A'$ is the right tilt of $\A$ with respect to the torsion pair 
\[
\T = \{X  \: | \: \Hom_{\A}(X,S)=0 \} \qquad \F = \langle S \rangle = \{ S^{\oplus n} \: | \: n \in \N_0\}
\]
\end{itemize}
\end{defn}

The relevance of simple tilts is that they occur precisely at the codimension $1$ components of the boundary of the space of stability conditions $U(\A)$ supported on a given heart $\A$ by \cite[Lemma 5.5]{Br2}.  Thus $\Stab(\D)$ is glued together from the $U(\A)$ according to the exchange graph.

\begin{defn}
 The exchange graph $\EG(\D)$ of $\D$ has vertices the set of hearts $\A \subset \D$ and an edge between any two hearts related by a simple tilt.  Define $\EG^0(\D)$ to be the connected component containing the standard heart $\A^0$.
\end{defn}
\noindent We compute the four simple tilts of the standard heart $\A^0$.
\begin{prop}
Denote by $E$ and $X$ the unique non-trivial extensions up to isomorphism of $S$ by $T$ and $T$ by $S[1]$ respectively.  Let $(A,B)_C$ denote the abelian category generated by two simple objects $A$ and $B$ having a unique up to isomorphism non-trivial extension $C$ of $B$ by $A$, so that the standard heart $\A^0 = (T,S)_E$.  Then
 \begin{align*}
  & R_S(\A^0) = (S[1], T)_{X[1]} &  L_T(\A^0) &= (T[-1], S)_{X} \\
  & R_T(\A^0) = (E, T[1])_S      &  L_S(\A^0) &= (S[-1], E)_T
 \end{align*}
Moreover the tilted hearts are obtained from $\A^0$ by applying the following autoequivalences.
\begin{align*}
 & R_S(\A^0) = (\Phi_T \Phi_S \Phi_T) [3] \;(\A^0) &  L_T(\A^0) &= ((\Phi_T \Phi_S \Phi_T) [3])^{-1} \;(\A^0) \\
  & R_T(\A^0) = (\Phi_S \Phi_T) [2] \; (\A^0)      &  L_S(\A^0) &= ((\Phi_S \Phi_T) [2])^{-1} \;(\A^0) 
\end{align*}

\end{prop}

We will prove the statement about the left tilt at $T$, the rest being similar.  The torsion pair in this case is
\[
 \T = \langle T \rangle \qquad \F = \{X \in \A^0 \: | \: \Hom_{\A^0}(T,X)=0 \} = \langle S \rangle
\]
We will use the long exact sequence in cohomology with respect to the original t-structure $\A^0$, the groups being non-zero only in degrees $0$ and $1$.

\begin{lemma}
 $T[1]$ is simple in $L_T(\A^0)$
\end{lemma}
\begin{proof}
 Consider a short exact sequence in $L_T(\A^0)$
\[
 0 \to X \to T[-1] \to Y \to 0
\]
giving a long exact sequence in $\A^0$.
\[
 0 \to H^0(X) \to H^0(T[-1]) \to H^0(Y) \to H^1(X) \to H^1(T[-1]) \to H^1(Y) \to 0
\]
We have $H^0(T[-1]) = 0$ so $H^0(X) = 0$. Splitting the remaining 4-term exact sequence into two short exact sequences
\begin{align*}
 0 &\to H^0(Y) \to H^1(X) \to Z \to 0 \\
0 &\to Z \to T \to H^1(Y) \to 0
\end{align*}
$Z$ is either $0$ or $T$ as $T$ is simple in $\A^0$. But there are no non-zero maps from $H^0(Y) \in \F$ to $H^1(X) \in \T$ so $H^1(X) \isom Z$ so $X$ is either $0$ or $T[-1]$ and so $T[-1]$ is simple.

\end{proof}

\begin{lemma}
$S$ is simple in $L_T(\A^0)$.
\end{lemma}

\begin{proof}  As $H^1(S) =0$ we have as before
\begin{align*}
 0 &\to H^0(X) \to S \to Z \to 0 \\
 0 &\to Z \to H^0(Y) \to H^1(X) \to 0
\end{align*}
Thus as $S$ is simple in $\A^0$, $H^0(X)$ is either $0$ or $S$, and so $Z$ is either $S$ or $0$.  Then as there are no non-zero maps from $H^0(Y) \in \F$ to $H^1(X) \in T$, $H^1(X) = 0$ and so $S$ is simple in $L_T(\A^0)$.
\end{proof}

\noindent We remark that all four simple tilts of $\A^0$ are isomorphic to $\A^0$ so the above is the local structure of the exchange graph at any vertex of the connected component $\EG^0(\D)$.

\begin{defn}
Let $\Aut^0(\D)$ be the subquotient of $\Aut(\D)$ consisting of all autoequivalences preserving the connected component $\EG^0(\D)$ of the exchange graph modulo those acting trivially on it.
\end{defn}
\noindent We will see later that in fact $\Aut^0(\D)$ is the subquotient preserving the connected component $\Stab^0(\D)/ \C$ modulo those acting trivially on it.
\begin{prop}
 The vertices of the connected component $\EG^0(\D)$ of the exchange graph are a torsor for $\Aut^0(\D)$.
\end{prop}

\begin{proof}
 From the above computation every heart in $\EG^0(\D)$ can be obtained by applying an autoequivalence in $\langle \Phi_S, \Phi_T, [1]\rangle$ to the standard heart $\A^0$.  Thus $\Aut^0(\D)$ acts transitively on $\EG^0(\D)$ and acts freely by definition. 
\end{proof}

\begin{lemma} The centre of $\Sph(\D)$ is generated by $[-5]$.
\end{lemma}
\begin{proof}
 As $\Sph(\D) \isom \Br_3$ the centre is generated by $\Phi = (\Phi_S \Phi_T)^3$.  We compute $\Phi$ on $S$ and $T$
\begin{align*}
 S \mapsto X\phantom{[-2]} \mapsto T[-1]  \mapsto T[-3] \mapsto E[-3] \mapsto S[-3]  \mapsto S[-5] \\
 T \mapsto T[-2]           \mapsto E[-2]  \mapsto S[-2] \mapsto S[-4] \mapsto X[-4]  \mapsto T[-5] 
\end{align*}
 Thus $\Phi = [-5]$ in $\Aut^0(\D)$.
\end{proof}

We note that $\Sph(\D)$ defines a subgroup of $\Aut^0(\D)$ isomorphic to $\Br_3$.  The generators $\Phi_S$ and $\Phi_T$ are composites of two autoequivalences corresponding to simple tilts, e.g. $\Phi_S^{-1} = (\Phi_T \Phi_S \Phi_T [3]) (\Phi_S \Phi_T [2])$ and so preserve the connected component of the exchange graph.  If an element of $\Sph(\D)$ acts trivially on $K(\D)$ then it belongs to the centre which we have just seen is generated by a non-trivial element of $\Aut^0(\D)$ so the only element of $\Sph(\D)$ acting trivially is the identity.

\begin{thm} The map $\Br_3 \to \Aut^0(\D)$ given by $(\sigma_1, \sigma_2) \mapsto (\Phi_S [1], \Phi_T [1])$ is an isomorphism. 
\end{thm}

\begin{proof}
As the exchange graph is an $\Aut^0(\D)$-torsor we know that $\Aut^0(\D) =\langle \Phi_S, \Phi_T, [1]\rangle$.  As the shift functor  commutes with the spherical twists, we find that $((\Phi_S [1])(\Phi_T [1]))^3 = [-5][6] = [1]$ so $\Aut^0(\D) = \langle \Phi_S [1], \Phi_T [1] \rangle$.  These two generators satisfy the braid relation as $\Phi_S, \Phi_T$ do.

Now consider a word $w$ in the generators $\Phi_S[1], \Phi_T[1]$ and their inverses which is equal to the identity of $\Aut^0(\D)$.  As $[1]$ is in the centre of $\Aut^0(\D)$, we have $\Phi_S^{n_1} \ldots \Phi_T^{n_k} = [-1]^{\sum n_i}$ in $\Sph(\D)$. By the above lemma the centre of $\Sph(\D)$ is generated by $(\Phi_S \Phi_T)^3 = [-5]$, so the right hand side is equal to $[-5]^{(\sum_i n_i)/5}$. As the braid relation is homogeneous, every element of $\Sph(\D)$ has a well-defined word length in the generators $\Phi_S$ and $\Phi_T$.  But applying the word length homomorphism gives $\sum_i n_i = \frac{6}{5} \sum_i n_i$ so $\sum_i n_i = 0$.  Thus the relations satisfied by the generators $\Phi_S [1], \Phi_T [1]$  of $\Aut^0(\D)$ are precisely those satisfied by the generators $\Phi_S, \Phi_T$ of $\Sph(\D)$.
\end{proof}

To complete the picture we show that $\Sph(\D)$ is a normal subgroup of index $5$.
\begin{prop}
There is a short exact sequence
\[
1 \to \Sph(\D) \to \Aut^0(\D) \to \Z/5\Z \to 1
\]
where the quotient map $l$ is the modulo $5$ word length map  in the generators $\Phi_S [1]$ and $\Phi_T [1]$
\end{prop}
\begin{proof}
$\Sph(\D)$ is in the kernel of $l$ as
\[
l(\Phi_S) = l(\Phi_S[1]) - l([1]) = 1-6 = 0
\]
Conversely the smallest power of $[1]$ in the kernel is $[5]= (\Phi_S \Phi_T)^{-3} \in \Sph(\D)$ and so as $\Aut^0(\D) = \langle \Phi_S, \Phi_T , [1] \rangle$ the kernel is contained in $\Sph(\D)$.
\end{proof}

\begin{rmk} 
 By Sabidussi's Theorem \cite[Thm 4]{Sab}, $\EG^0(\D)$ is isomorphic to the Cayley graph of the braid group $\Br_3$ with respect to the generators $\Delta = (\Phi_T \Phi_S \Phi_T) [3]$ and $\Sigma=(\Phi_S \Phi_T) [2]$ which give the simple tilted hearts.  Indeed this gives an alternative presentation of $\Br_3$ \cite[Sect 1.14]{KT}
\[
\langle \Sigma, \Delta \: | \: \Sigma^3 = \Delta^2 \rangle
\]
The quotient of $\EG^0(\D)$ by $\Sph(\D)$ is the $A_2$ cluster exchange graph which is isomorphic to the Cayley graph of $\Z/{5\Z}$.  This recovers a special case of a result of Keller and Nicolas \cite[Thm 5.6]{K}.
\end{rmk}

\section{Stability conditions}

In this section we prove Theorem \ref{stab}.  We derive the Picard-Fuchs equations satisfied by the periods of the family of meromorphic differentials $\lambda$ on the fibres $E$ of the universal family of framed elliptic curves $\E \to \widetilde{\M}$.  Identifying the lattices $H_1(E, \Z) \isom K(\D)$, the image in $\Pb\Hom(K(\D), \C)$ of a certain branch of the period map is a double of the Schwarz triangle with angles $(\pi, \pi/3, \pi/2)$.  We show that this coincides with the image under the local homeomorphism $\bar{Z}: \Stab^0(\D) / \C \to \Pb \Hom(K(\D), \C)$ of a fundamental domain for the action of $\Aut^0(\D)/ \Z \isom \PSL(2, \Z)$ on $\Stab^0(\D) / \C$.  We use our understanding of the exchange graph of $\D$ to lift the period map to our desired biholomorphism $f: \widetilde{\M} \to \Stab^0(\D) / \C$.

\begin{defn}
On an elliptic curve $y^2 = z^3 + az+ b$ define the meromorphic differential $\lambda = y \: dz$
\end{defn}
This has a pole of order $6$ at the point at infinity and double zeroes at each of the three other branch points of $y$.  This is the divisor of the function $y^2$.  It is the unique differential up to scale with this property as the above divisor has degree zero.

Define the coordinates $j$ and $u$ on $\widetilde{\M}$ by 
\begin{equation} \label{u}
 J = 1728/j \qquad j = 4u(1-u)
\end{equation}
where $J$ denotes the usual $J$-invariant.  We note that the family of differentials $\lambda= \sqrt{z^3 -3z + (4u - 2)} \: dz$ satisfy $2 \del_u \lambda = \omega$, where $\omega = dz/y$ is the family of holomorphic differentials on $\widetilde{\M}$.  Using this we show that the periods of $\lambda$ satisfy hypergeometric equations in $u$ and $j$.
\begin{defn}
A hypergeometric differential equation is a second order ordinary differential equation on $\Pb^1$ of the form
\[
 w(1-w) f'' + (\gamma - (\alpha + \beta - 1) w )f' - \alpha \beta w = 0
\]
with $\alpha, \beta, \gamma \in \R$.  
\end{defn}

It has regular singularities at $0$, $\infty$ and $1$ with exponents
\[
\lambda = 1 - \gamma \qquad \mu = \alpha - \beta \qquad \nu = \gamma - \alpha - \beta
\]

\begin{lemma}
 The periods of $\lambda$ satisfy the hypergeometric equation in $j$ with exponents $(1,\frac{1}{3},\frac{1}{2})$
\end{lemma}

\begin{proof}
 Suppose the periods of $\lambda$ satisfy the hypergeometric equation in $u$
\[
 u(1-u) \del^2_u f + (\gamma - (\alpha + \beta - 1) j ) \del_u f - \alpha \beta f = 0
\]
Taking the derivative with respect to the dependent variable $u$, we find that the periods of $\omega$ must satisfy 
\[
  u(1-u) \del^2_u f + (1-2u)\del_u f + (\gamma - (\alpha + \beta - 1) u )\del_u f -(\alpha +\beta-1) f - \alpha \beta f = 0
\]
which is hypergeometric of the form
\[
  u(1-u) \del^2_u f + ((\gamma+1) - ((\alpha+1) + (\beta+1) - 1) u ) \del_u f - (\alpha+1) (\beta+1) f = 0
\]

It is well-known the periods of $\omega$ satisfy the hypergeometric equation in $j$ with exponents $(0, \frac{1}{3}, \frac{1}{2})$.  Then by the quadratic transformation law for the change of variable given above \cite[Eq 2]{V}, they satisfy the hypergeometric equation in $u$ with exponents $(0, \frac{1}{3}, 0)$.  By the above computation, the periods of $\lambda$ satisfy the hypergeometric equation with exponents $(1, \frac{1}{3},1)$ and so reversing the change of variable gives the result.
\end{proof}

\begin{rmk}
The coordinate transformation (\ref{u}) defines a double cover $B \to M_{1,1}$ of the coarse moduli space of elliptic curves.  There is a family of elliptic curves on $B$ whose total space is the complement of the three singular fibres of types $(I_1, I_1, II^*)$ over $u=0,1$ and $\infty$ respectively of a rational elliptic surface $\Sigma_u \to \Pb^1_u$.  This is the smooth part of Hitchin's fibration of the moduli space of meromorphic $\SU(2)$-Higgs bundles on $\Pb^1_z$ with a single pole of order $4$ at $z =\infty$ whose leading term is nilpotent.  The meromorphic differential $\lambda$ is the Seiberg-Witten differential of this integrable system, that is the exterior derivative of $\lambda$ defines a holomorphic symplectic form on $\Sigma$.

In fact $\Sigma$ is a hyperk\"{a}hler manifold \cite{W}, which was studied in \cite[Sect 9.3.3]{GMN}.  In another complex structure $\Sigma$ is isomorphic to the moduli space of flat $\SL(2, \C)$-connections on $\Pb^1_z$ with a single pole at $z= \infty$ of Katz invariant $5/2$.  This complex manifold was studied in \cite{S,vdPS} as the moduli space of initial conditions of the first Painlev\'{e} equation (cf Remark \ref{painleve}).  Its image under the Riemann-Hilbert map is an affine cubic surface which is isomorphic as a complex variety to the cluster algebra of $A_2$.
\end{rmk}

Now consider the moduli space of elliptic curves $\M \isom \Pb(2,3) \less \{ \circ\}$ where $\circ$ is the point corresponding to $j= 0$.  We make branch cuts on $\M$ along the line $\Im(j) = 0$ between $\circ$ and each of the $\Z_2$ and $\Z_3$ orbifold points $\times$, $*$ at $j=1, \infty$.  We deduce the image of this branch of the period map $p$ of $\lambda$ from the Schwarz triangle theorem.
\begin{thm} \cite[p 206]{N}
 Suppose $f_1$, $f_2$ are linearly independent solutions to the hypergeometric equation with exponents $(\lambda, \mu, \nu)$.  Suppose further that their ratio $s= f_1/ {f_2}$ restricted to the upper-half plane $\h \subset \Pb^1 \less \{0, \infty, 1\}$ is an injection.  Then $s$ maps $\h$ biholomorphically onto the interior of a curvilinear triangle $\Delta_{\lambda, \mu, \nu}$ of angles $(\lambda \pi, \mu \pi, \nu \pi)$.
\end{thm}
The image is determined up to a M\"{o}bius map and so specified uniquely by the positions of the three vertices of the triangle $\Delta$.  By the Schwarz reflection principle we have
\begin{cor}
 The image $\lozenge = p(\M)$ is the double of the curvilinear triangle $\Delta_{1, \frac{1}{3}, \frac{1}{2}}$ along the edge connecting the image of the two orbifold points $\times$ and $*$.
\end{cor}

We now define a fundamental domain $V=V(\Abar^0)$ for the action of $\Aut^0(\D) / \Z$ on $\Stab^0(\D) / \C$ which maps bijectively under the local homeomorphism $\bar{Z}$ to $\lozenge$.  Although the vertices of the quotient of the exchange graph $\overline{\EG^0}(\D) = \EG^0(\D) / {\Z[1]}$ are indeed an $\Aut^0(\D)/ \Z$-torsor, the notion of a projective stability condition $\sigmabar \in \Stab^0(\D)/\C$ being supported at a given vertex $\Abar$ of $\overline{\EG^0}(\D)$ is not a priori well-defined.  This is because points of $\Stab^0(\D)$ in the same $\C$-orbit can be supported on different hearts, even modulo the shift functor.   We \emph{define} $\sigmabar$ to be supported on $\Abar$ using the following width function. 

\begin{defn}
 Define the width $\varphi$ of a stability condition $\sigma = (Z, \A) \in \Stab(\D)$ 
\[
 \varphi(\sigma) =  \phi^+(\sigma) - \phi^-(\sigma) 
\]
 where $\phi^+(\sigma)$ and $\phi^-(\sigma)$ denote the maximal and minimal phases respectively of an object in $\A$.
\end{defn}

\noindent The width is the angle of the image under $Z$ of the cone $C(\A) \subset K(\A)$ generated by classes of objects in $\A$.

\begin{defn}
We say that $\sigmabar \in \Stab^0(\D) / \C$ is supported on $\Abar$ if the width function is minimised on a lift $\A$ of $\Abar$.
\end{defn}

\noindent Note that $\sigmabar$ is supported on more than one $\Abar$ where the width function is minimised on more than one such $\Abar$.  We will write $V(\Abar) \subset \Stab^0(\D) / \C$ for the subset supported uniquely on $\Abar$, whose closure $\bar{V}(\Abar)$ is the subset supported on $\Abar$.

\begin{prop}
 $V = V(\Abar^0)$ is the interior of a fundamental domain for the action on $\Aut^0(\D) / \Z$ on $\Stab^0(\D) / \C$
\end{prop}

\begin{proof}
 As the vertices of $\overline{\EG}^0(\D)$ are an $\Aut^0(\D)/ \Z$-torsor, every point in the set $T=\coprod_{\Abar} V(\Abar)$  belongs to a unique $V(\Abar)$.  The points $\sigmabar$ in $\bar{V} \less V$ lie on the three codimension 1 walls pictured below where $\sigmabar$ is also supported on some  other $\Abar$ for some simple tilt $\A$ of $\A^0$.    These walls of the $V(\Abar)$ are locally finite as there is only one other wall intersecting $\bar{V}$, namely $\bar{V}(L_S(\Abar^0)) \cap \bar{V}(R_T(\Abar^0))$.  Thus the closure $\bar{T} = \coprod_{\Abar} \bar{V}(\Abar)$.  But $\bar{T}$ is clearly open and so is the entire connected component $\Stab^0(\D)/ \C$. 
\end{proof}

\begin{rmk} The above proof shows that an autoequivalence $\Phi$ which preserves $\overline{\EG}^0(\D)$ preserves the connected component $\Stab^0(\D)/ \C$.  Also if $\Phi$ acts trivially on $\overline{\EG}^0(\D)$ then $\Phi$ fixes the central charge $\bar{Z}$ and so $\Phi$ acts trivially on $\Stab^0(\D)/ \C$. 
\end{rmk}

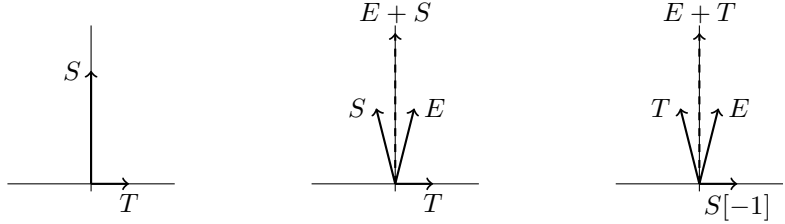
\begin{figure}[h]
 \begin{center}
\begin{tikzpicture}
 \draw (-1.1,0) -- (1.1,0);
 \draw (0,-0.1) -- (0,2.1);
 \draw [->,thick](0,0) -- (0,1.5) node [anchor=east]{$S$};
 \draw [->,thick] (0,0) -- (0.5,0) node [anchor = north]{$T$};

 \draw (2.9,0) -- (5.1,0);
 \draw (4,-0.1) -- (4,2.1);
 \draw [->,thick](4,0) -- (4.25,1) node [anchor=west]{$E$};
 \draw [->,thick](4,0) -- (3.75,1) node [anchor=east]{$S$};
 \draw [->,thick] (4,0) -- (4.5,0) node [anchor = north]{$T$};
 \draw [->,thick, dashed](4,0) -- (4,2) node [anchor=south]{$E+S$};

 \draw (6.9,0) -- (9.1,0);
 \draw (8,-0.1) -- (8,2.1);
 \draw [->,thick](8,0) -- (8.25,1) node [anchor=west]{$E$};
  \draw [->,thick](8,0) -- (7.75,1) node [anchor=east]{$T$};
 \draw [->,thick] (8,0) -- (8.5,0) node [anchor = north]{$S[-1]$};
\draw [->,thick, dashed](8,0) -- (8,2) node [anchor=south]{$E+T$};

\end{tikzpicture}
\end{center}
\caption{Typical stability conditions on the three boundary components of $V(\bar{\A^0})$.  The first $\bar{V}(\Abar^0) \cap \bar{V}(R_S(\Abar^0)) = \bar{V}(\Abar^0) \cap \bar{V}(L_T(\Abar^0))$ occurs where the only stable objects are $S$ and $T$.  The other two $\bar{V}(\Abar^0) \cap \bar{V}(R_T(\Abar^0))$ and $\bar{V}(\Abar^0) \cap \bar{V}(R_T(\Abar^0))$ lie in the region where $S, T$ and $E$ are stable.}
\end{figure}
\noindent This means that $\Stab^0(\D)/ \C$ is glued together from the $V(\Abar)$ according to the quotient of the exchange graph $\overline{\EG}^0(\D)$ just as $\Stab^0(\D)$ is glued from the $U(\A)$ according to $\EG^0(\D)$.
\begin{prop}  The image of $V$ under the map $\bar{Z}$ is $\lozenge$.
\end{prop}
\begin{proof}  The boundary of $V$ consists of stability conditions supported on one of the three walls which we picture below, whose image under $\bar{Z}$ is the boundary of $\lozenge$.
\end{proof}

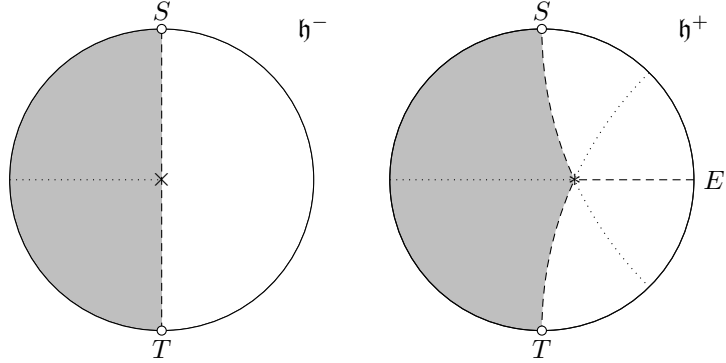
\begin{figure}[h]
\begin{center}
\begin{tikzpicture}
\pgfmathparse{{2}}\global\let\r\pgfmathresult

\draw (0,0)  circle (\r cm) ;
\pgfmathparse{{\r*tan(67.5)}}\global\let\rt\pgfmathresult
\pgfmathparse{(\rt - sqrt(\rt^2-\r^2))}\global\let\x\pgfmathresult

\begin{scope}
\clip (0,0) circle (\r cm) ;
\fill[lightgray] (0,0) circle (\r cm) ;

\fill[white] (\rt, \r) circle (\rt cm) ;
\fill[white] (\rt, -\r) circle (\rt cm) ;
\draw (0,0) circle (\r cm) ;
\draw [dotted](-\r,0) -- (\x,0);
\draw [densely dashed](\r,0) -- (\x,0);

\begin{scope}
\clip  (\x, -\r) rectangle (\r, \r);
\draw [dotted](\rt, \r) circle (\rt cm) ;
\draw [dotted](\rt, -\r) circle (\rt cm) ;
\end{scope}

\begin{scope}
\clip  (-\r, -\r) rectangle (\x, \r);
\draw [densely dashed](\rt, \r) circle (\rt cm) ;
\draw [densely dashed](\rt, -\r) circle (\rt cm) ;
\end{scope}

\end{scope} 
\draw (0,0) circle (\r cm) ;
\fill[white] (0, \r) circle (\r*0.03 cm)  ;
\draw (0, \r) circle (\r*0.03 cm) ;
\fill[white] (0, -\r) circle (\r*0.03 cm) ;
\draw (0, -\r) circle (\r*0.03 cm) ;

\draw (0,\r) node [anchor = south] {$S$} ;
\draw (0, -\r) node [anchor=north] {$T$} ;
\draw (\r,0) node [anchor = west] {$E$} ;
\draw (\x,0) node {$*$} ;
\draw (\r,\r) node {$\h^+$} ;

\begin{scope}
 \clip (-2*\r-1, 0) circle (\r cm);
 \fill[lightgray] (-3*\r-1,-\r) rectangle (-2*\r-1, \r);

\draw (-2*\r-1, 0) circle (\r cm);
\draw [densely dashed] (-2*\r-1,-\r) -- (-2*\r-1,\r);

\draw (-\r-1,\r) node {$\h^-$};
\draw (-2*\r-1, 0) node {$\times$} ;
\draw [dotted] (-3*\r-1,0) -- (-2*\r-1,0);
\end{scope}
 \draw (-2*\r-1, 0) circle (\r cm);
\draw (-\r-1,\r) node {$\h^-$} ;
\fill[white] (-2*\r-1, \r) circle (\r*0.03 cm) ;
\draw (-2*\r-1, \r) circle (\r*0.03 cm) ;
\fill[white] (-2*\r-1, -\r) circle (\r*0.03 cm) ;
\draw (-2*\r-1, -\r) circle (\r*0.03 cm) ;
\draw (-2*\r-1,\r) node [anchor = south] {$S$} ;
\draw (-2*\r-1, -\r) node [anchor=north] {$T$} ;
\end{tikzpicture}

\end{center}
\caption{The fundamental domain $V(\bar{\A}^0) \isom \lozenge$ under the map $\bar{Z}: \Stab^0(\D)/\C \to \Pb^1$.  We picture $\Pb^1 = \h^+ \cup \h^-$ as the union of two discs where the imaginary part of the coordinate $\bar{Z}$ is positive and negative respectively.  They are glued along the line $\bar{Z} \in \R$, which is the image of all walls of marginal stability in $\Stab(\D)/ \C$.  The region $\h^-$ where only two objects are stable contains the first wall passing through the image of $\times$.  The region $\h^+$ contains the other two walls of $V(\bar{\A}^0)$ which meet at the image of $*$.  We label points on the boundary by the object whose central charge vanishes there.}
\end{figure}
{
\begin{proof}[Proof of Theorem \ref{stab}] 
Using the identification $V(\A) \isom \lozenge = p(\M)$, we can extend the branch of the period map to a map $f:\widetilde{\M} \to \Stab(\D) / \C$ by equivariance.  We only have to check continuity on the boundary of $\M$, i.e. the action of the monodromy on $H_1(E, \Z)$ on crossing one of the two branch cuts in either direction is identical to the action of the four simple tilts on $K(\D)$.  But these both act by 
\[
\left ( \begin{array}{cc}  0 & -1 \\
                           1 & 0 \\
 \end{array} \right )
\qquad \left ( \begin{array}{cc} 0 & 1 \\
                                -1 & 1\\
 \end{array} \right )
\]
and their inverses.
\end{proof}
}
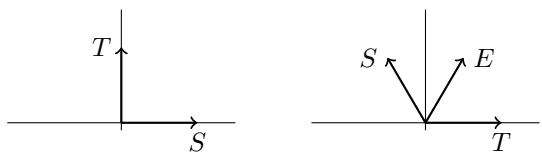
\begin{figure}[h] \label{orb}
\begin{center}
\begin{tikzpicture}
 \draw (-1.5,0) -- (1.5,0);
 \draw (0,-0.1) -- (0,1.5);
 \draw [->,thick](0,0) -- (0,1) node [anchor=east]{$T$};
 \draw [->,thick] (0,0) -- (1,0) node [anchor = north]{$S$};

 \draw (2.5,0) -- (5.5,0);
 \draw (4,-0.1) -- (4,1.5);
 \draw [->,thick](4,0) -- (4.5,0.86) node [anchor=west]{$E$};
 \draw [->,thick](4,0) -- (3.5,0.86) node [anchor=east]{$S$};
 \draw [->,thick] (4,0) -- (5,0) node [anchor = north]{$T$};
\end{tikzpicture}
\end{center}
\caption{The image of the  $\Z/2$-  and $\Z/3$-orbifold points $\times$ and $*$}
\end{figure}

We denote by $L^\times$ the total space of the $\C^*$-bundle of non-zero holomorphic differentials on $\widetilde{M}$.  It is isomorphic to the complement of the discriminant locus in the space $\C^2_{a,b}$ of cubic polynomials $z^3 + az + b$.  The fundamental group of $L^\times$ is isomorphic to the braid group $\Br_3$ as the discriminant locus describes the trefoil knot.
\begin{cor}
 
 There is a biholomorphic map
\begin{center}
\begin{tikzpicture}[node distance=2cm, auto]
  \node (A) {$\widetilde{L^\times}$};
  \node (B) [right of=A] {$\Stab^0(\D)$};
  \node (D) [below of=B] {$\Hom(K(\D),\C)$};
;
  \draw[->] (A) to node {$F$} (B);
  \draw[->] (A) to node [anchor=east] {$(\int_\alpha \lambda, \int_\beta \lambda) \: \:$} (D);
    \draw[->] (B) to node {$(Z(S),Z(T))$} (D);
\end{tikzpicture}
\end{center}
 lifting the periods of the differential $\lambda$.  It is equivariant with respect to the actions of $\Br_3$ on the left by deck transformations and on the right by $\Aut(\D)$.
\end{cor}
\begin{proof}
 We can lift the map $f: \widetilde{\M} \to \Stab^0(\D) / \C$ to the desired $F$ by equivariance with respect to the $\C$-actions on both sides.  It is a bijection as both $\C$-actions are free, and holomorphic as it is locally given by the periods of $\lambda$.  We know that the two braid groups act identically on $K(\D)$ via their maps to $\PSL(2,\Z)$ and so define identical actions on the $\C^*$-bundle $L^\times$.  Also the actions of the central subgroup $\Z \subset \Br_3$ are identical by construction as it acts as $\Z \subset \C$.  But given these data the actions are determined by a group homomorphism $\PSL(2, \Z) \to \Z$ giving a lifting of the $\Br_3$-action on the $\C^*$ bundle $L^\times$  factoring through $\PSL(2, \Z)$ to the universal cover.  As the only such homomorphism is the trivial one the two $\Br_3$ actions are identical. 
\end{proof}

\bibliographystyle{hplain}
\bibliography{end}

\begin{thebibliography}{10}

\bibitem{B}
Tom Bridgeland.
\newblock Stability conditions on triangulated categories.
\newblock {\em Ann. of Math. (2)}, 166(2):317--345, 2007.

\bibitem{Br2}
Tom Bridgeland.
\newblock Spaces of stability conditions.
\newblock In {\em Algebraic geometry---{S}eattle 2005. {P}art 1}, volume~80 of
  {\em Proc. Sympos. Pure Math.}, pages 1--21. Amer. Math. Soc., Providence,
  RI, 2009.

\bibitem{GMN}
Davide Gaiotto, Gregory~W. Moore, and Andrew Neitzke.
\newblock Wall-crossing, {H}itchin {S}ystems, and the {WKB} approximation,
  2009, arXiv:0907.3987.

\bibitem{KT}
Christian Kassel and Vladimir Turaev.
\newblock {\em Braid groups}, volume 247 of {\em Graduate Texts in
  Mathematics}.
\newblock Springer, New York, 2008.
\newblock With the graphical assistance of Olivier Dodane.

\bibitem{K}
Bernhard Keller.
\newblock On cluster theory and quantum dilogarithm identities, 2011,
  arXiv:1102.4148.

\bibitem{KYZ}
Bernhard Keller, Dong Yang, and Guodong Zhou.
\newblock The {H}all algebra of a spherical object, 2008, arXiv:0810.5546.

\bibitem{N}
Zeev Nehari.
\newblock {\em Conformal mapping}.
\newblock McGraw-Hill Book Co., Inc., New York, Toronto, London, 1952.

\bibitem{Sab}
Gert Sabidussi.
\newblock On a class of fixed-point-free graphs.
\newblock {\em Proc. Amer. Math. Soc.}, 9:800--804, 1958.

\bibitem{S}
Hidetaka Sakai.
\newblock Rational surfaces associated with affine root systems and geometry of
  the {P}ainlev\'e equations.
\newblock {\em Comm. Math. Phys.}, 220(1):165--229, 2001.

\bibitem{ST}
Paul Seidel and Richard Thomas.
\newblock Braid group actions on derived categories of coherent sheaves.
\newblock {\em Duke Math. J.}, 108(1):37--108, 2001.

\bibitem{vdPS}
Marius van~der Put and Masa-Hiko Saito.
\newblock Moduli spaces for linear differential equations and the {P}ainlev\'e
  equations.
\newblock {\em Ann. Inst. Fourier (Grenoble)}, 59(7):2611--2667, 2009.

\bibitem{V}
Raimundas Vid{\=u}nas.
\newblock Algebraic transformations of {G}auss hypergeometric functions.
\newblock {\em Funkcial. Ekvac.}, 52(2):139--180, 2009.

\bibitem{W}
Edward Witten.
\newblock Gauge theory and wild ramification.
\newblock {\em Anal. Appl. (Singap.)}, 6(4):429--501, 2008.

\end{thebibliography}
\end{document}